\documentclass[reqno,12pt]{amsproc}
\usepackage{amssymb}
\usepackage{amsmath}
\usepackage{amsfonts}
\usepackage{microtype}
\usepackage[canadian]{babel}
\usepackage{xcolor}
\usepackage{geometry}
\usepackage{enumitem}
\usepackage{tikz-cd}

\definecolor{myurlcolor}{rgb}{0,0,0.3}
\definecolor{mycitecolor}{rgb}{0,0.3,0}
\definecolor{myrefcolor}{rgb}{0.3,0,0}
\usepackage[pagebackref,draft=false]{hyperref}
\hypersetup{colorlinks,
linkcolor=myrefcolor,
citecolor=mycitecolor,
urlcolor=myurlcolor}

\usepackage{cleveref}

\newcommand{\beq}{\begin{equation}}
\newcommand{\eeq}{\end{equation}}
\newcommand{\N}{\mathbb{N}}
\newcommand{\Nplus}{\mathbb{N}_{> 0}}			
\newcommand{\Z}{\mathbb{Z}}
\newcommand{\Q}{\mathbb{Q}}
\newcommand{\R}{\mathbb{R}}
\newcommand{\Rplus}{\mathbb{R}_{> 0}}			

\newcommand{\eps}{\varepsilon}

\newcommand{\asympgeq}{\gtrsim}			
\newcommand{\Spec}[1]{\mathsf{Spec}(#1)}		

\newcommand{\aovs}[1]{\mathsf{aovs}(#1)}	
\newcommand{\Nbhd}{\mathcal{N}}			

\theoremstyle{plain}
\newtheorem{dummy}{Dummy}[section]
\newtheorem{thm}[dummy]{Theorem}\Crefname{thm}{Theorem}{Theorems}
\newtheorem{lem}[dummy]{Lemma}\Crefname{lem}{Lemma}{Lemmas}
\newtheorem{prop}[dummy]{Proposition}\Crefname{prop}{Proposition}{Propositions}
\newtheorem{cor}[dummy]{Corollary}\Crefname{cor}{Corollary}{Corollaries}
\Crefname{conj}{Conjecture}{Conjectures}
\Crefname{qstn}{Question}{Questions}
\newtheorem{defn}[dummy]{Definition}\Crefname{defn}{Definition}{Definitions}
\newtheorem{prob}[dummy]{Problem}\Crefname{prob}{Problem}{Problems}
\theoremstyle{remark}
\newtheorem{ex}[dummy]{Example}\Crefname{ex}{Example}{Examples}
\newtheorem{rem}[dummy]{Remark}\Crefname{rem}{Remark}{Remarks}
\Crefname{note}{Note}{Notes}
\numberwithin{equation}{section}

\Crefformat{enumi}{#2#1#3}

\allowdisplaybreaks

\protected\def\verythinspace{%
	\ifmmode
		\mskip0.5\thinmuskip
	\else
		\ifhmode
			\kern0.08334em
		\fi
	\fi}
\renewcommand{\,}{\verythinspace}

\let\originalleft\left
\let\originalright\right
\renewcommand{\left}{\mathopen{}\mathclose\bgroup\originalleft}
\renewcommand{\right}{\aftergroup\egroup\originalright}

\usepackage{todonotes}

\begin{document}

\title{A generalization of Strassen's Positivstellensatz}

\author{Tobias Fritz}

\address{Perimeter Institute for Theoretical Physics, Waterloo, Canada}
\email{tfritz@pitp.ca}

\keywords{}

\subjclass[2010]{Primary: 06F25, 16Y60; Secondary: 12J15, 14P10.}

\thanks{\textit{Acknowledgements.} Many thanks to P\'eter Vrana for pointing out an error in a previous version of this paper. We also thank Max Daniel, Richard K\"ung, Rostislav Matveev, Tim Netzer, Luciano Pomatto, Konrad Schm\"udgen, Markus Schweighofer, Matteo Smerlak, Volker Strassen, Arleta Szko{\l}a, Omer Tamuz and P\'eter Vrana for useful discussions and feedback, as well as David Handelman and Terence Tao for discussion at \href{https://mathoverflow.net/questions/314378/when-can-a-function-be-made-positive-by-averaging}{mathoverflow.net/questions/314378/when-can-a-function-be-made-positive-by-averaging}. We also thank the referee for useful comments which have improved the manuscript substantially. Most of this work was conducted while the author was with the Max Planck Institute for Mathematics in the Sciences, which we thank for its outstanding research environment.}

\begin{abstract}
Strassen's Positivstellensatz is a powerful but little known theorem on preordered commutative semirings satisfying a boundedness condition similar to Archimedeanicity. It characterizes the relaxed preorder induced by all monotone homomorphisms to $\R_+$ in terms of a condition involving large powers. Here, we generalize and strengthen Strassen's result. As a generalization, we replace the boundedness condition by a polynomial growth condition; as a strengthening, we prove two further equivalent characterizations of the homomorphism-induced preorder in our generalized setting. 
\end{abstract}

\maketitle

\tableofcontents


\section{Introduction}

The subject of Positivstellens\"atze in real algebraic geometry has a long history tracing back to Hilbert's 17th problem and Artin's proof that every nonnegative real multivariate polynomial is a sum of squares of real rational functions. Since then, many further Positivstellens\"atze have been proven, with sums of squares often at centre stage~\cite{marshall,PD,scheiderer},~\cite[Chapter~12]{schmudgen}. This has resulted in a surprising variety of applications to probability, optimal control, and other areas~\cite{lasserre}.

A result going in a different direction is Strassen's Positivstellensatz from~\cite{strassen}, which we recall below as \Cref{spss}. Here, sums of squares do not appear to play an important role. Rather, a central concept is the \emph{asymptotic preordering} on a preordered semiring defined in terms of taking large powers of the semiring elements. Perhaps due to this difference, and because Strassen's result was developed around its intended application to the computational complexity of matrix multiplication, this Positivstellensatz does not seem to be well-known in the real algebraic geometry community. Let us give some comments on why we think that Positivstellens\"atze of Strassen's type are important.

Preordered semirings are ubiquitous structures coming up in many areas of mathematics; in fact, they secretly enjoy a pervasiveness similar to that of the concept of category. To wit, the collection of all mathematical structures of a certain kind often forms a preordered semiring.
Historically, this type of observation has become apparent primarily in the study of vector bundles, whose isomorphism classes form a semiring under direct sum and tensor product, leading to the development of $K$-theory. One can equip this semiring with a preorder structure by declaring one vector bundle to be greater than another if the first contains an isomorphic copy of the second. Similar constructions are interesting in many other settings as well: they make sense whenever one has a class of mathematical objects with well-behaved analogues of direct sum and tensor product constructions, in such a way that the tensor product distributes over the direct sum, and such that there is a sensible notion of when one object is contained another object or perhaps forms a quotient or subquotient, giving a preorder relation. If this preorder relation is compatible with the direct sums and tensor products, then we obtain a preordered semiring. Thus Positivstellens\"atze for abstract semirings, such as Strassen's, should be expected to be essential and useful tools in analyzing the structure of semirings of this type, giving criteria for when a large power of one object contains an isomorphic copy of a large power of another object.

However, the inherent boundedness requirement (\Cref{strassen_preordered}) in Strassen's Positivstellensatz makes it inapplicable in many such situations. At the most basic level, Strassen's Positivstellensatz does not even apply to the polynomial preordered semiring $\N[X]$ with the coefficientwise order. This is why we develop a generalization of Strassen's Positivstellensatz as \Cref{main}, where the boundedness condition on a preordered semiring $S$ of characteristic zero is replaced by a polynomial growth condition. We will present further strengthenings in upcoming work. We plan to explore a number of applications to semirings of mathematical objects, as described in the previous paragraph, in future research.

More technically, suppose that $S$ is a preordered semiring of polynomial growth and satisfying $1 \ge 0$. Then our \Cref{main} provides three different characterizations of when given nonzero elements $x,y\in S$ satisfy $f(x) \geq f(y)$ for all order-preserving semiring homomorphisms $S \to \R_+$. One of these characterizations recovers Strassen's Positivstellensatz; the other two are new even in the case where Strassen's stronger boundedness assumption holds. They enable us to make connections with more conventional Positivstellens\"atze involving sums of squares (\Cref{compare}).

\subsection*{Relation to the existing literature.} We now outline the relation between our Positivstellensatz and its proof and the existing literature in real algebraic geometry. A reader familiar with this literature may want to take the following remarks as a starting point for studying the present paper.

\begin{enumerate}
	\item A central theme for us is the replacement of Strassen's Archimedeanicity-type assumption (\Cref{strassen_preordered}) by a polynomial growth condition (\Cref{univ_defn}). An assumption of this type has previously been used by Marshall in~\cite{marshall2} and subsequently by Schweighofer in~\cite[Theorem~5.1]{schmudgen} for a generalization of Schm\"udgen's Positivstellensatz which does not require compactness of the semialgebraic set involved. Thus while there is strong conceptual similarity here to our definition, the technical details of the definition itself and the way in which the respective growth condition is used in the proofs seem quite different.
\item One of the most basic applications of our Positivstellensatz is to the semiring of polynomials $\R_+[X_1,\ldots,X_d]$ equipped with the coefficientwise preorder (\Cref{poly}). In this case, we obtain a result which characterizes nonnegativity of a polynomial on the positive orthant. This result is strictly weaker than P\'olya's Positivstellensatz, which is also concerned with coefficientwise positivity~\cite[p.~57]{ineqs}. However, our \Cref{ex_subset} goes beyond P\'olya's Positivstellensatz by characterizing positivity on arbitrary subsets of the orthant $\R_+^d$ in terms of coefficientwise positivity.
\item The most commonly considered types of ordered algebraic structures in real algebraic geometry are rings, equipped either with a \emph{preprime}~\cite[Definition~5.4.1]{PD} or with a \emph{quadratic module}~\cite[p.~4/5]{PD}. Our approach in terms of preordered semirings subsumes both of these situations as special cases, resulting in \Cref{qm,kd-like}. This is similar to Marshall's approach involving $T$-modules for a preprime $T$~\cite[Chapter~5]{marshall}, which also comprises both preprimes and quadratic modules, although Marshall's main focus is on the Archimedean case.

	More precisely, there is an equivalence of structures between preordered semirings $S$ which satisfy $1 \ge 0$ as well as order-cancellativity on the one hand, and rings $R$ equipped with a preprime $T$ and a $T$-module $M$ in Marshall's sense such that $R = T - T$ on the other hand (\Cref{equiv}). Since we do not require our preordered semirings to be order-cancellative, our approach is even more general than Marshall's.
\item The proof of \Cref{semifield_case} crucially relies on the fact that the cone $C^*$ appearing in there is the union $\bigcup_{\eps > 0} C^*_\eps$, where each subcone $C^*_\eps$ is a face of $C^*$ with a compact cap. The fact that this is possible relies in our usage of the locally convex topology induced by the sets $\Nbhd_\eps$. We suspect that this topology is related to the result of Lasserre and Netzer on the density of sums of squares polynomials inside nonnegative polynomials~\cite{LN}.
\item One of the central ingredients of our proof of \Cref{semifield_case} is showing that a monotone additive map is extremal among additive monotone maps if and only if it is a scalar multiple of a homomorphism. Results of this type have been known for a long time~\cite{BLP,BSS},~\cite[Proposition~12.33]{schmudgen}.
\item Although our \Cref{main} is not concerned with sums of squares, they do naturally come up at~\eqref{convex_order} and after. Something similar has been known to happen for Archimedean preprimes, which approximately contain every sum of squares~\cite[Proposition~2]{krivine}.
\end{enumerate}

\section{A new Positivstellensatz for preordered semirings}

A \emph{semiring} is a set equipped with a commutative monoid structure called \emph{multiplication}, which distributes over another commutative monoid structure called \emph{addition}. Thus a semiring is like a ring, except in that additive inverses do generally not exist; due to the absence of \emph{n}egatives, semirings are also sometimes called \emph{rigs}.

If $S$ and $T$ are semirings, then a \emph{semiring homomorphism} from $S$ to $T$ is a map $f : S \to T$ which preserves addition and multiplication, $f(x + y) = f(x) + f(y)$ and $f(xy) = f(x) f(y)$, as well as the neutral elements, $f(0) = 0$ and $f(1) = 1$.

\begin{defn}
A semiring $S$ has \emph{characteristic zero} if the unique homomorphism $\N \to S$ is injective.
\end{defn}

For example, $\N$ itself or $\R_+$ are semirings of characteristic zero. The protagonists of this paper are semirings which carry an additional order structure.

\begin{defn}
	A \emph{preordered semiring} is a semiring $S$ with a preorder relation ``$\geq$'' such that
	\[
		x\geq y \qquad \Longrightarrow \qquad x + z \geq y + z \quad \&\quad xz \geq yz, 
	\]
\end{defn}

Although we will not consider this part of the definition, our results will also assume that $1 \geq 0$ in $S$. In this case, we necessarily have $x \geq 0$ for all $x\in S$, since $1 \geq 0$ implies $1 \cdot x \geq 0 \cdot x$.

If $S$ and $T$ are preordered semirings, then a homomorphism $f : S \to T$ is \emph{monotone} if it is order-preserving, meaning that $x \geq y$ in $S$ implies $f(x) \geq f(y)$ in $T$. Further, $f$ is an \emph{order embedding} if $x \geq y$ is equivalent to $f(x) \geq f(y)$. For example, the unique homomorphism $\N \to S$ is monotone for every preordered semiring $S$ with $1 \ge 0$, but not necessarily an order embedding.

\begin{rem}
	\label{equiv}
	We now explain the relation between preordered semirings and Marshall's \emph{$T$-modules}~\cite[Chapter~5]{marshall}; readers not interested in this comparison may directly move ahead to \Cref{strassen_preordered}.

	A \emph{preprime} $T$ in a ring $R$ is a subset $T \subseteq R$ which is a subsemiring,\footnote{Marshall also requires $\Q_+ \subseteq T$, which is convenient for his purposes and essentially without loss of generality. But since this extra condition is not relevant in our context, we have decided to omit this extra condition, as have previous authors (e.g.~\cite{dubois}).}
	\[
		T + T \subseteq T, \qquad TT \subseteq T, \qquad 0,1 \in T.
	\]
	Given a preprime $T$, a \emph{$T$-module} $M$ is another subset $M \subseteq R$ such that
	\[
		M + M \subseteq M, \qquad TM \subseteq M, \qquad 0,1 \in M.
	\]
	For example if $T \subseteq R$ is the subsemiring of sums of squares, then the $T$-modules are exactly the \emph{quadratic modules} in $R$.

	Given a $T$-module $M$, a preprime $T$ becomes a preordered semiring if we define $x \ge y$ to hold if and only if $x - y \in M$; the relevant conditions are straightforward to verify. This preordered semiring satisfies $1 \ge 0$, and by definition is \emph{order-cancellative}, i.e.~it satisfies
	\[
		x + z \ge y + z \quad \Longrightarrow \quad x \ge y.
	\]
	If the preprime $T$ satisfies $R = T - T$, then $T$ considered as a preordered semiring contains all the information about $M$: an arbitrary element $x \in R$ is in $M$ if and only if, upon writing it as $x = x_1 - x_2$ with $x_1, x_2 \in T$, we have $x_1 \ge x_2$.

	Conversely, suppose that $S$ is a preordered semiring with $1 \ge 0$ and that $S$ is order-cancellative. Then let us write $x \approx y$ in $S$ if $x \ge y$ and $y \ge x$; it is easy to see that $\approx$ is an equivalence relation, and that addition and multiplication of equivalence classes is well-defined. We also have that $x + z \approx y + z$ implies $x \approx y$. Thus the quotient $S / \!\approx$ is a preordered semiring which is both cancellative and order-cancellative. We can now take $R := S \otimes \Z$, defined via the Grothendieck construction as the set of all formal differences of elements of $S / \!\approx$. The image of $S$ in $R$ defines a preprime $T \subseteq R$. If we moreover define
	\[
		M := \{x - y \in R \mid x,y \in S \:\land\: x \ge y\},
	\]
	then it is easy to see that $M$ is a $T$-module. In this way, we can associate a preprime $T$ and a $T$-module to every preordered semiring.

	It is also routine to show that the two constructions from the previous two paragraphs are inverses of each other in the following sense. If $M$ is a $T$-module in a ring $R$ and $R = T - T$, then the preordered semiring $T$ is such that $T \otimes \Z \cong R /(M \cap -M)$ as rings, with $T$-module given by $M / (M \cap -M)$. Thus up to dividing by the (largely irrelevant) ``null space'' $M \cap -M$, our construction recovers the $T$-module from the associated preordered semiring. Conversely if $S$ is an order-cancellative preordered semiring with $1 \ge 0$, then $S/\!\approx$ is order isomorphic to the preprime $S/\!\approx \:\subseteq\: S \otimes \Z$. Thus also in this case, our construction recovers the preordered semiring from the associated $T$-module, at least up to identifying elements which are equivalent in the order anyway.

	Hence in this sense, ordered-cancellative preordered semirings with $1 \ge 0$ and $T$-modules for preprimes $T \subseteq R$ with $R = T - T$ are essentially equivalent structures. Since our results assume $1 \ge 0$, the main difference between our setup and Marshall's setup of $T$-modules is that we do not require order-cancellativity.
\end{rem}

We now return to the theory of preordered semirings. Strassen has considered the following condition, even if not under that name:

\begin{defn}
\label{strassen_preordered}
A preordered semiring $S$ with $1 \ge 0$ is \emph{Strassen preordered} if the homomorphism $\N\to S$ is an order embedding, and for every nonzero $x\in S$ there is $\ell\in\N$ with
\beq\label{boundedness}
	\ell x \geq 1,\qquad \ell \geq x.
\eeq
\end{defn}

We think of these two inequalities as a boundedness requirement: every element is both upper bounded and lower bounded by a scalar. It is thus intuitively analogous to Archimedeanicity for ordered rings and quadratic modules. The extra condition that $\N\to S$ must be an order embedding is stronger than having characteristic zero; for example, $\N$ itself equipped with the preorder in which all elements are comparable has characteristic zero. But as we will see, having characteristic zero is perfectly sufficient for Strassen's Positivstellensatz to hold. The boundedness requirement~\eqref{boundedness} is the essential part of Strassen's condition.

\begin{rem}
	Note that Strassen's original formulation of~\eqref{boundedness} rather goes like this: for every two nonzero $x,y\in S$, there is $k\in\N$ with $k x \geq y$. This trivially implies our formulation~\eqref{boundedness} upon taking $x = 1$ or $y = 1$ and $\ell = k$. Conversely, our formulation implies Strassen's: if $\ell x \geq 1$ and $\ell' \geq y$, then we obtain $(\ell \ell') x \geq y$.
\end{rem}

We now present Zuiddam's improved version~\cite[Theorem~2.2]{zuiddam} of Strassen's Positivstellensatz~\cite[Corollary~2.6]{strassen}.

\begin{thm}[Strassen, Zuiddam]
\label{spss}
Let $S$ be a Strassen preordered semiring. Then for any $x,y\in S$, the following are equivalent:
\begin{enumerate}
\item $f(x) \geq f(y)$ for every monotone semiring homomorphism $f : S \to \R_+$.
\item\label{asymp_order_spss} For every $\eps > 0$ there are $k,n\in\Nplus$ such that $k\leq \eps n$ and
\[
	2^k x^n \geq y^n.
\]
\end{enumerate}
\end{thm}

\begin{rem}
\label{asymptotic_no}
Since property~\ref{asymp_order_spss} involves a limit of large powers, the preorder characterized by the two equivalent conditions has been called the \emph{asymptotic preorder}. Our upcoming strengthening in \Cref{spss_better}, consisting of the additional equivalent conditions \ref{closure_order_new} and \ref{cat_order_new}, suggests that this may not be such a fitting name after all. From our perspective, it behaves more like a change of base to $\R_+$, combined with a topological closure.
\end{rem}

As mentioned before, our main result is not only a strengthening of \Cref{spss}, but also a generalization to a wider class of preordered semirings. Namely, we relax the above boundedness condition to a polynomial growth condition, where polynomial growth is measured relative to a fixed element:

\begin{defn}
\label{univ_defn}
Let $S$ be a preordered semiring with $1 \ge 0$. An element $u\in S$ with $u\geq 1$ is:
\begin{enumerate}
\item \emph{polynomially universal} if for every nonzero $x\in S$ there is a polynomial $p\in\N[X]$ with
\beq
\label{poly_univ}
	p(u)\, x \geq 1,\qquad p(u) \geq x.
\eeq
\item \emph{power universal} if for every nonzero $x\in S$ there is $k\in\N$ with
\beq
\label{power_univ}
	u^k\, x \geq 1, \qquad u^k \geq x.
\eeq
\end{enumerate}
If $S$ has either type of universal element, then we say that $S$ is \emph{of polynomial growth}.
\end{defn}

In the context of ordered rings, conditions of this type have (as far as we know) first been considered by Marshall~\cite{marshall2}.

\begin{rem}
\label{poly_vs_power}
Clearly a power universal element is trivially also polynomially universal. Conversely, if $u$ is polynomially universal, then $2u$ is power universal. This is because $u\geq 1$ implies that $p(u) \leq p(1) \, u^{\deg(p)}$, and both the scalar $p(1)$ and the power $u^{\deg(p)}$ are dominated by a suitably large power of $2u$. Thus the distinction between polynomially universal and power universal elements is rather inessential.
\end{rem}

\begin{rem}
Strassen's boundedness condition of \Cref{strassen_preordered} holds if and only if $u = 1$ is polynomially universal, or equivalently if any natural number $u \in \N$ with $u \geq 2$ is power universal.
\end{rem}

\begin{rem}
\label{check_gens}
The set of all nonzero $x\in S$ which satisfy the polynomial universality condition is closed under addition and multiplication. It is therefore enough to check the condition on a set of generating elements.
\end{rem}

Our main result is as follows:

\begin{thm}
\label{main}
Let $S$ be a preordered semiring of characteristic zero and satisfying $1 \ge 0$, and suppose that $u \in S$ is polynomially universal. Then for every nonzero $x,y\in S$, the following are equivalent:
\begin{enumerate}
\item\label{hom_order} $f(x) \geq f(y)$ for every monotone semiring homomorphism $f : S \to \R_+$.
\item\label{closure_order} For every $r\in\R_+$ and $\eps > 0$, there exist a polynomial $p\in\N[X]$, a nonzero element $z\in S$ and $m\in\Nplus$ such that $p(r) \leq \eps\, m$ and
\beq
\label{closure_order_eq}
	m\, z\, x + p(u) \, z \geq m\, z\, y.
\eeq
\item\label{cat_order} For every $r\in\R_+$ and $\eps > 0$, there exist a polynomial $p\in\N[X]$, a nonzero element $z\in S$ and $m\in\Nplus$ such that $p(r) \leq (1 + \eps) m$ and
\beq
\label{cat_order_eq}
	p(u)\, z\, x \geq m\, z\, y.
\eeq
\item\label{asymp_order} For every $r\in\R_+$ and $\eps > 0$, there exist a polynomial $p\in\N[X]$ and $n,m\in\Nplus$ such that $p(r) \leq (1 + \eps)^n m$ and
\beq
\label{asymp_order_ineq}
	p(u)\, x^n \geq m \, y^n. 
\eeq 
\end{enumerate}
\end{thm}

We present the proof in \Cref{proofs}. 

\begin{rem}
The equivalence of~\ref{hom_order} and~\ref{asymp_order} is somewhat reminiscent of the classical fact that for a ring $R$, an element of $R$ vanishes when considered as a residue-field-valued function on $\Spec{R}$ if and only if it is nilpotent.
\end{rem}

\begin{rem}
	In \ref{closure_order}--\ref{asymp_order}, the polynomial $p$ is monotonically increasing in $r$. Therefore large $r$ is the regime of interest. For the same reason, these conditions encode pointwise convergence: as $\eps \to 0$ and $r\to \infty$, the polynomial $m^{-1} p$ from \ref{closure_order} converges to zero pointwise; and similarly, the polynomial $m^{-1} p$ from \ref{cat_order} and the function $\sqrt[n]{m^{-1} p}$ from \ref{asymp_order} must converge to $1$ in both the pointwise topology on functions and the topology of compact convergence as $\eps \to 0$ and $r \to \infty$.
\end{rem}

\begin{rem}
\label{Ralg}
In the case that $S$ happens to be an algebra over $\R_+$, with $\R_+\subseteq S$ carrying the usual order, then we can clearly take all polynomials $p$ in the statement to be in $\R_+[X]$ and restrict to $m = 1$ in \ref{closure_order}--\ref{asymp_order}. Furthermore, every monotone homomorphism $S\to\R_+$ is then automatically $\R_+$-linear, due to the monotone case of Cauchy's functional equation. The same is true with $\Q_+$ in place of $\R_+$.
\end{rem}

\begin{rem}[{\cite[Lemma~5.4]{ocm}}]
	In some cases, condition~\ref{asymp_order} may hold with $\eps = 0$, i.e.~there may be $n\in\N$ with $x^n \geq y^n$. Then also condition~\ref{cat_order} holds with $\eps = 0$, since $z := \sum_{k=1}^n x^{k-1} y^{n-k}$ is a nonzero element satisfying $xz \geq yz$,
	\[
		xz = x^n + \sum_{k=1}^{n-1} x^k y^{n-k} \geq y^n + \sum_{k=1}^{n-1} x^k y^{n-k} = yz.
	\]
\end{rem}

\begin{rem}
	Since one is typically only interested in nonzero elements in applications, we do not think that the restriction to \emph{nonzero} $x,y\in S$ in \Cref{main} poses much of a problem.
\end{rem}

As a pure strengthening of Strassen's \Cref{spss}, we therefore have the following:
 
\begin{cor}
\label{spss_better}
Let $S$ be a Strassen preordered semiring. Then for nonzero $x,y\in S$, the following are equivalent:
\begin{enumerate}
	\item\label{hom_order_new} $f(x) \geq f(y)$ for every monotone semiring homomorphism $f : S \to \R_+$.
\item\label{closure_order_new} For every $\eps > 0$, there are $m,n\in\Nplus$ and nonzero $z\in S$ such that $m \leq \eps n$ and
\[
	n\, z\, x + m\, z \geq n\, z\, y.
\]
\item\label{cat_order_new} For every $\eps > 0$, there are $m,n\in\Nplus$ and nonzero $z\in S$ such that $\frac{m}{n} \leq 1 + \eps$ and
\[
	m\, z\, x \geq n\, z\, y.	
\]
\item\label{asymp_order_new} For every $\eps > 0$ there are $k,n\in\Nplus$ such that $k\leq \eps n$ and
\[
	2^k \, x^n \geq y^n.
\]
\end{enumerate}
Moreover, if $f(x) > f(y)$ for every monontone semiring homomorphism $f : S \to \R_+$, then there are nonzero $z,w\in S$ with $zx + w \geq zy + w$.
\end{cor}

\begin{proof}
	Taking $u = 1$ and $r = 1$ in \Cref{main}, it is easy to see that the various conditions given there take the form that we have written down here, since all polynomials can be taken to be scalars without loss of generality. For \ref{asymp_order_new}, this is not quite obvious: what we have is that for all $\eps > 0$ there are $n,m,\ell \in \Nplus$ such that $\ell \le (1 + \eps)^n m$ and
	\[
		\ell \, x^n \ge m \, y^n.	
	\]
	Upon increasing $n$ if necessary, we can assume that $\ell$ is a multiple of $m$, so that we have $n,m,j \in \Nplus$ such that $j \le (1 + \eps)^n$ and
	\[
		m \, j \, x^n \ge m \, y^n.
	\]
	But then we obtain
	\[
		m \, j^k \, x^{nk} \ge m \, y^{nk} \ge y^{nk}
	\]
	for all $k \in \Nplus$ by induction on $k$. Since $m j^k \le (1 + 2\eps)^{nk}$ for large enough $k$, the claim follows upon replacing $\eps$ by $\frac{\eps}{2}$ throughout.

	For the final claim, compactness of the asymptotic spectrum~\cite{strassen} implies that $f(x) - f(y)$ is bounded below by a positive scalar, say $\ell^{-1}$ for $\ell\in\Nplus$. Then with $x' := \ell x$ and $y' := \ell y + 1$, we have $f(x') \geq f(y')$ for all $f$, so that we can apply~\ref{closure_order_new} with $\eps = 1$, resulting in $m \leq n$ and
	\[
		n\, z\, \ell\, x + m\, z \geq n\, z\, (\ell\, y + 1) \geq n\, z \, \ell\, y + m\, z,
	\]
	which is of the desired form.
\end{proof}

\begin{rem}[Zuiddam, personal communication]
	The proof of our \Cref{main}, and hence also of our \Cref{spss_better}, uses tools from functional analysis. It is also possible to prove at least \Cref{spss_better} by purely algebraic means using the method of Zuiddam based on considering maximal extensions of the given preorder and observing that every maximal Strassen preorder is total~\cite[Chapter~2]{zuiddam_thesis}\footnote{This method of proof is very common in real algebra, and e.g.~analogous to the method of Becker and Schwartz for proving the Positivstellensatz of Krivine--Kadison--Dubois~\cite{BS}.}. More concretely, defining the relaxed preorder relation $x\succcurlyeq y$ to hold if condition~\ref{hom_order_new} is satisfied, then one can verify that~\cite[Lemma~2.4]{zuiddam_thesis} holds, and these are precisely the relevant properties of the relaxed preorder $\succcurlyeq$ that Zuiddam uses in order to prove the equivalence with~\ref{hom_order_new}. The same applies to the equivalence of condition~\ref{cat_order_new} with~\ref{hom_order_new}; and also to the equivalence of~\ref{asymp_order_new} with~\ref{hom_order_new}, which is the case considered by Zuiddam's.

	It is thus also a natural question whether our \Cref{main} has a purely algebraic proof. We will indeed present a further generalization of this result with a purely algebraic proof in upcoming work.
\end{rem}

\begin{rem}
	If the preoder on $S$ is moreover multiplicatively cancellative, then conditions~\ref{closure_order_new} and~\ref{cat_order_new} simplify further. In particular,~\ref{closure_order_new} becomes equivalent to: for every $n\in\N$, we have $n x + 1 \geq ny$.

	One may wonder about the relation between \Cref{spss_better} and the classical Positivstellensatz of Krivine--Kadison--Dubois~\cite{BS,dubois,krivine}, and in particular whether \Cref{spss_better} for order-cancellative $S$ immediately follows from a version of the latter upon applying a suitable version of it to the ordered ring $S \otimes \Z$ as in \Cref{equiv}. However, this is not the case, even with Marshall's general version of the theorem~\cite[Theorem~5.4.4]{marshall}. The reason is that upon the translation to $T$-modules, the assumptions of \Cref{spss_better} guarantee that the relevant $T$-module $M$ is Archimedean, but there is no guarantee that $T$ itself is Archimedean or even just a quasi-preordering in the sense of~\cite[Definition~5.4.2]{marshall}. Another way to see that \Cref{spss_better} is a different result is to note that every version of the Krivine--Kadison--Dubois theorem implies an algebraic inequality without any further ``padding'' like multiplication by a nonzero $z$ or considering large powers. And indeed the latter cannot be expected to hold under the assumptions of \Cref{spss_better}, as one can see upon considering $S := \N[X]$ preordered with respect to the smallest order-cancellative semiring preorder which satisfies $1 \le X \le 2$. Then the elements $x = 2 X^2 + 3$ and $y = 4 X$ satisfy\footnote{One way to see this is to note that the map $\N[X] \to \N$ given by $p \mapsto 2p(1) - p(0)$ is monotone with respect to the given preorder.} $x \not\ge y$ despite $f(x) > f(y)$ for all monotone homomorphisms $f$, where the latter are the evaluation maps on the interval $[1,2]$.
\end{rem}

\begin{rem}
	There may be interesting consequences of \Cref{spss_better} for graph theory, along the lines of the application of Strassen's Positivstellensatz to graph theory due to Zuiddam~\cite{zuiddam}.
\end{rem}

To finish up our general developments, we derive a \emph{rate formula} for preordered semirings analogous to the rate formula for ordered commutative monoids from~\cite{ocm}. The following definition is the analogue of~\cite[Definition~8.16]{ocm}.

\newcommand{\Rreg}[2]{R_{\mathrm{reg}}\left({#1}\to{#2}\right)}

\begin{defn}
	For a preordered semiring $S$ with $1 \ge 0$ and a polynomially universal element $u$ and nonzero $x,y\in S$, a number $\lambda\in\R_+$ is a \emph{regularized rate} from $x$ to $y$ if for every $r\in\R_+$ and $\eps > 0$ there are $m,n,\ell\in\Nplus$ with $\frac{m}{n} \geq \lambda - \eps$ and $p\in\N[X]$ with $p(r) \leq (1 + \eps)^n \ell$ such that
\[
	p(u) \, x^n \geq \ell \, y^m.
\]
We write $\Rreg{x}{y}$ for the largest regularized rate from $x$ to $y$ (which may be infinite).
\label{rate_def}
\end{defn}

Essentially by definition, the set of regularized rates is closed as a subset of $\R_+$, so that the largest regularized rate indeed exists (and is finite if the set of regularized rates is bounded).

We then have a semiring analogue of the rate formula given in~\cite[Theorem~8.24]{ocm}.

\begin{cor}
\label{rate_formula}
If $S$ satisfies the hypotheses of \Cref{main} and $x,y \geq 1$, then regularized rates can be computed as
\[
	\Rreg{x}{y} = \inf_f \frac{\log f(x)}{\log f(y)},
\]
where the infimum ranges over all monotone semiring homomorphisms $f : S \to \R_+$, with the convention $\frac{0}{0} := \infty$.
\end{cor}

The improvement over the rate formula of~\cite{ocm} is that the functions that need to be optimized over only comprise those monotone maps which preserve \emph{both} kinds of algebraic structure, which is a much stronger condition than merely the preservation of one of them. The reason that the logarithm appears is that the maps $\log f$ are homomorphisms from the multiplicative monoid of $S$ to the monoid $(\R,+)$, thereby putting them into the picture of~\cite{ocm}.

\begin{proof}
We show that $\lambda\in\R_+$ is a regularized rate if and only if $\log f(x) \geq \lambda \log f(y)$ for all monotone semiring homomorphisms $f : S \to \R_+$, from which the claim follows immediately.

In one direction, if $\lambda$ is a regularized rate, then we choose $m,n\in\Nplus$ and $p$ for given $\eps > 0$ as in Definition~\ref{rate_def}, using $r = f(u)$ as before. The assumptions then show that $(1 + \eps)^n \ell f(x)^n \geq \ell f(y)^m$ for any given $f$, and therefore $\log(1 + \eps) + \log f(x) \geq (\lambda - \eps) \log f(y)$. The claim $\log f(x) \geq \lambda \log f(y)$ follows in the limit $\eps \to 0$.

Conversely, suppose that $\log f(x) \geq \lambda \log f(y)$ for all $f$, and let $r\in\R_+$ and $\eps > 0$ be given. Let us choose an arbitrary rational $\frac{m}{n} \in [\lambda-\eps,\lambda]$. Then we have $f(x^n) \geq f(y^m)$ for all monotone semiring homomorphisms $f : S \to \R_+$, so that we get $p\in\N[X]$ and $k,\ell\in\Nplus$ with $p(r) \leq (1 + \eps)^k \ell$ and $p(u) \, x^{kn} \geq \ell \, y^{km}$ from \Cref{main}. Since in particular $p(r) \leq (1 + \eps)^{k n} \ell$ and $\frac{km}{kn} \geq \lambda - \eps$, we are done.
\end{proof}

We end the discussion of our main results with a remark and an open problem.

\begin{rem}
\label{asymp_mono}
One may hope that if $u$ is power universal, then the polynomial in~\eqref{asymp_order_ineq} can be taken to be a monomial $u^k$ with $k$ sublinear in $n$. However, this is not possible; the following example was pointed out to us by Peter Vr\'ana. Consider again the semiring $\N[X]$, now with repect to the semiring preorder in which $p \ge q$ if and only if $p(r) \ge q(r)$ for all large enough $r$. Then this makes $\N[X]$ even totally preordered with power universal element $u := 2 + X$, but there is no monotone homomorphism $\N[X] \to \R_+$. Hence condition~\ref{hom_order} of \Cref{main} holds for any two nonzero polynomials in $\N[X]$. It is easy to see that the other conditions of \Cref{main} likewise hold. However, the candidate property of $u^k p^n \ge q^n$ for large enough $n$ and $k$ sublinear in $\N$ does clearly not hold, e.g.~with $p = 1$ and $q = X$.

One may also hope to have an analogue of the final statement in \Cref{spss_better} under the assumptions of \Cref{main}. For example, if $f(x) > f(y)$ for every monotone semiring homomorphism $f : S \to \R_+$, then does it follow that there is nonzero $z \in S$ with $z\, x \geq z\, y$ on the nose, i.e.~without an additional correcting factor as in~\eqref{cat_order_eq}? The same example shows that this is not the case in general, even under stronger assumptions like $f(x) > f(u) f(y)$ for all $f$.
\end{rem}

\begin{prob}
\label{zu}
In~\eqref{closure_order_eq} and~\eqref{cat_order_eq}, under what conditions can one take $z$ to be a polynomial in $u$?
\end{prob}

\section{Proof of \Cref{main}}
\label{proofs}

A semiring is a \emph{semifield} if every nonzero element $x$ has a multiplicative inverse $x^{-1}$. A semifield of characteristic zero automatically is a $\Q_+$-algebra, and we will freely use scalar multiplication by positive rationals. Correspondingly, the Greek variables $\eps$ and $\alpha$ that we use in the following denote scalars which are assumed to be rational. Furthermore, in order for $u \geq 1$ in a preordered semifield $F$ with $1 \ge 0$ to be power universal, it is enough to check the upper bound condition, namely that for every $x\in F$ there is $k\in\N$ with $u^k \geq x$. This is the property that we use below, after possibly replacing without loss of generality the polynomially universal element $u$ by a power universal one, such as $2u$ (\Cref{poly_vs_power}).

The following result is the crucial stepping stone towards \Cref{main}.

\begin{thm}
\label{semifield_case}
Let $F$ be a preordered semifield of characteristic zero and satisfying $1 \ge 0$, and suppose that $u\in F$ is polynomially universal. Then for every nonzero $x,y\in F$, the following are equivalent:
\begin{enumerate}
\item\label{hom_order_sf} $f(x) \geq f(y)$ for every monotone semiring homomorphism $f : S \to \R_+$.
\item\label{closure_order_sf} For every $\eps > 0$, there is $m\in\N$ such that
\beq
\label{approx_order_ineq}
	 x + \sum_{j=0}^m \eps^{j+1} u^j \geq y.
\eeq
\item\label{closure_order_sf2} For every $r\in\R_+$ and $\eps > 0$, there is a polynomial $p\in\Q_+[X]$ such that $p(r) \leq \eps$ and
\beq
\label{approx_order_ineq2}
	x + p(u) \geq y.
\eeq
\item\label{cat_order_sf} For every $r\in\R_+$ and $\eps > 0$, there is a polynomial $p\in\Q_+[X]$ such that $p(r) \leq 1 + \eps$ and
\beq
\label{cat_order_sf_ineq}
	p(u)\, x \geq y.
\eeq
\end{enumerate}
\end{thm}

\begin{proof}
Condition~\ref{closure_order_sf} easily implies property~\ref{closure_order_sf2}, since~\eqref{approx_order_ineq} contains the geometric series: for given $r$, we choose $\eps < r^{-1}$ and $m$ such that~\eqref{approx_order_ineq} holds, and take $p := \sum_{j=0}^m \eps^{j+1} X^j$. Then we have $p(r) = \sum_{j=0}^m \eps^{j+1} r^j \leq \frac{\eps}{1-\eps r}$, which indeed can be made arbitrarily small.

Given~\ref{closure_order_sf2}, choose $k\in\N$ with $u^k \geq x^{-1}$, so that
\[
	(1 + u^k p(u)) \, x \geq x + x\, x^{-1}\, p(u) \geq y.
\]
Thus $\hat{p} := 1 + X^k p$ has the desired properties, since $\hat{p}(r) \leq 1 + r^k p(r)$ can be taken to be arbitrarily close to $1$ by assumption.
	
Assuming~\ref{cat_order_sf}, we apply a given $f$ as in~\ref{hom_order_sf} on both sides of the given inequality and use the homomorphism property in order to pull $f$ into the polynomial, resulting in
\[
	p(f(u))\, f(x) \geq f(y).
\]
With $r := f(u)$ and choosing $p$ as a function of $\eps$, we get $(1 + \eps)\, f(x) \geq f(y)$, and therefore $f(x) \geq f(y)$ in the limit $\eps \to 0$.

The difficult part of the proof is the final implication from \ref{hom_order_sf} to \ref{closure_order_sf}. For now, we assume that $F$ is order-cancellative, meaning that $w + x \geq w + y$ implies $x \geq y$; we will treat the general case at the end of the proof. We work with the $\Q$-vector space $V := F\otimes \Q = F - F$, and equip it with the induced preordering, which is characterized by a positive cone $C\subseteq V$ via the usual correspondence between ordered vector spaces and convex cones. For given $\eps > 0$, consider the set
\[
	\Nbhd_\eps := \left\{ \: x \in F \:\bigg|\: \exists m\in\N, \: x \leq \sum_{j=0}^m \eps^{j+1} u^j \: \right\}.
\]
We use the set $\Nbhd_\eps - \Nbhd_\eps$ as $\eps$ varies as a basis of neighbourhoods of zero in $V$. Since every $x\in F$ satisfies $x \leq u^k$ for suitable $k$, we get $x\in \eps^{-(k+1)} \Nbhd_\eps$, and it follows that $\Nbhd_\eps - \Nbhd_\eps$ is absorbent. Each set $\Nbhd_\eps - \Nbhd_\eps$ is also easily seen to be absolutely convex by convexity of $\Nbhd_\eps$. These sets form a local basis at zero since $\Nbhd_\eps \subseteq \Nbhd_{\eps'}$ for $\eps \leq \eps'$, and $\Nbhd_{\eps/2} + \Nbhd_{\eps/2} \subseteq \Nbhd_\eps$. We therefore have a locally convex topology on $V$. 

We now characterize the closure of the preorder in this topology. We have $x \geq y$ in the closure if and only if for every $\eps$, there is $x'\geq y'$ in $F$ with $x - x'\in (\Nbhd_\eps - \Nbhd_\eps)$ and $y - y' \in (\Nbhd_\eps - \Nbhd_\eps)$, which implies in particular that there are $m,n\in\N$ such that
\[
	x + \sum_{j=0}^m \eps^{j+1} u^j \geq x', \qquad y' + \sum_{j=0}^n \eps^{j+1} u^j \geq y.
\]
Therefore, assuming without loss of generality $n=m$,
\[
	x + 2 \sum_{j=0}^n \eps^{j+1} u^j \geq x' + \sum_{j=0}^m \eps^{j+1} u^j \geq y' + \sum_{j=0}^m \eps^{j+1} u^j \geq y,
\]
which implies that $x$ and $y$ satisfy the condition of~\ref{closure_order_sf}, with $2\eps$ in place of $\eps$. Conversely, if~\ref{closure_order_sf} holds, then it is easy to see that $x \geq y$ in the closure. Thus condition~\ref{closure_order_sf} characterizes the closure of the preorder in the locally convex topology generated by the sets $\Nbhd_\eps - \Nbhd_\eps$, and the associated positive cone is $\overline{C}$, the closure of the original positive cone $C$.

Now consider the dual space $V^*$, by which we mean the set of all $\Q$-linear continuous maps $V \to \R$ equipped with the weak-$*$ topology. Then $V^*$ contains a closed convex cone $C^*$, the set of all these functionals that are monotone with respect to the original preorder on $F$, or equivalently nonnegative on $C$. We claim that $C^*$ is the closed conical hull of its extreme rays. For given $\eps > 0$, consider the subcone
\[
	C^*_\eps := \{ f \in C^* \mid f(\Nbhd_\eps) \textrm{ is bounded} \}.
\]
This set is a face of $C^*$, since it is a subcone, and if $f = f_1 + f_2$ is in $C^*_\eps$ with $f_1,f_2\in C^*$, then also $f_1,f_2\in C^*_\eps$ due to monotonicity of $f_1$ and $f_2$ and the boundedness assumption on $f$.

The set
\[
	\{ f \in C^*_\eps \mid f(\Nbhd_\eps) \subseteq [0,1] \}
\]
is a cap of $C^*_\eps$, in the sense that it is a convex subset of $C^*_\eps$ containing zero as well as a positive scalar multiple of every point of $C^*_\eps$. It is compact due to Tychonoff's theorem and the boundedness assumption, using that $\Nbhd_\eps - \Nbhd_\eps$ is absorbent. By Krein-Milman, we can therefore conclude that $C^*_\eps$ is the closed convex hull of its extreme points. Since $C^* = \bigcup_\eps C^*_\eps$ writes $C^*$ as a union of faces, it follows that also $C^*$ itself is the closed conical hull of its extreme rays.
The essential conceptual idea of the proof comes now, where we will show that the extreme rays of $C^*$ are precisely\footnote{Technically, it would be enough for the proof to show that every extreme ray is a multiple of a monotone homomorphism. But the proof of the converse is instructive as well and may be of use in other contexts, which is why we nevertheless include it and prove the equivalence.} the (unique $\Q$-linear extensions of) the monotone semiring homomorphisms $F \to \R_+$. The equivalence between~\ref{closure_order_sf} and~\ref{hom_order_sf} then follows by the previous parts of the proof and the Hahn-Banach theorem for locally convex spaces: if $x \geq y$ does not hold in the closure of the order, then there is a continuous monotone linear functional which witnesses this; we can always find such a witness that is an extreme ray of $C^*$ and therefore a monotone homomorphism, contradicting assumption~\ref{hom_order_sf}.

So suppose that $f : F \to \R_+$ is monotone, $\Q$-linear, $\Nbhd_\eps$-bounded, and extremal with these properties. Fix nonzero $a\in F$. Then for every $x\in F$,\footnote{Here we are assuming that $1 + a \neq 0$, so that $1 + a$ is indeed invertible. For if $1 + a = 0$, then $1 \geq 0$ implies $0 = 1 + a \geq a$. Since $a \geq 0$ as well, this implies $0 \geq a^2 = 1$. But then also $0 \geq x$ for all $x\in F$, making all four conditions of \Cref{semifield_case} trivially true. Hence we can assume $1 + a \neq 0$ without loss of generality.}
\[
	f(x) = f \left( \frac{1}{1+a} \cdot x \right) + f \left( \frac{a}{1+a} \cdot x \right),
\]
and when considered as a function of $x$, each term is again monotone, $\Q$-linear, and continuous. (The latter requires showing that for every $b\in F$, the multiplication map $b\cdot - : V \to V$ is $\Nbhd_\eps$-continuous, which goes like this: if $0 \leq x \leq \sum_{j=0}^m \eps^{j+1} u^j$, then also $0 \leq bx \leq \sum_{j=0}^m \eps^{j+1} u^{j+k} \leq \eps^{-k} \sum_{j=0}^{m+k} \eps^{j+1} u^j$, for $k\in\N$ such that $b\leq u^k$.) Therefore each term is a scalar multiple of $f$ by the extremality assumption. Since we can assume $f$ to be nonzero, it follows that also each term on the right is nonzero (for some $x$). Thus each term is even a scalar multiple of the other: there is $r > 0$ such that for all $x\in F$,
\[
	f\left( \frac{ax}{1+a} \right) = r \, f\left( \frac{x}{1+a} \right).
\]
Since we might as well use $x(1+a)$ in place of $x$, we conclude that $f(ax) = r f(x)$ for all $x\in F$. Taking $x = 1$ shows that $f(1) > 0$, because $a$ was arbitrary and $f$ was assumed to be nonzero. Thus we have shown that $f(1) > 0$ for every extremal ray $f\in C^*_\eps$. By virtue of Choquet's theorem applied to $C^*_\eps$, we can even conclude that $f(1) > 0$ for every nonzero $f\in C^*$, which we will use below.

By the above, we have $f(a) = f(a\cdot 1) = r f(1)$, so that $r = f(1)^{-1} f(a)$, and generally $f(ax) = f(1)^{-1} f(a) f(x)$. Since $a\in F$ was arbitrary, this means that $x\mapsto f(1)^{-1} f(x)$ is a monotone semiring homomorphism, as was to be shown.

Before approaching the converse direction, we prove that for every nonzero $x\in F$, we have
\beq
\label{convex_order}
	x + x^{-1} \geq 2
\eeq
in the closure of the order, i.e.~according to condition~\ref{closure_order_sf}. Inspired by the concept of \emph{embezzlement} from~\cite{embezzling}, we work with the Laurent-polynomial expression
\[
	z := \sum_{j = -n}^{+n} \left( 1 - \frac{|j|}{n} \right) x^j
\]
for given $n\in\N$. Some computation shows that
\[
	(x + x^{-1})z - 2z = \frac{1}{n} \left( x^{-n} + x^n \right) - \frac{2}{n}.
\]
Since $x^{-n} + x^n \geq 0$, we have therefore
\[
	\frac{2}{n} z^{-1} + x + x^{-1} \geq 2.
\]
Because $z \geq 1$ due to the $j = 0$ term, we obtain
\[
	\frac{2}{n} + x + x^{-1} \geq 2,
\]
so that~\eqref{convex_order} follows in the limit $n\to \infty$. 

Next, we can also write~\eqref{convex_order} as $x^2 + 1 \geq 2x$. Applying this inequality to $\alpha^{-1} x$ in place of $x$ shows that $x^2 + \alpha^2 \geq 2 \alpha x$ for any rational $\alpha > 0$, again in the closure of the order. So for any nonzero $f\in C^*$, we therefore have
\[
	f(x^2) - 2 \alpha f(x) + \alpha^2 f(1) \geq 0,
\]
due to $\Q$-linearity and monotonicity. By continuity in $\alpha$, this inequality holds for all real $\alpha \geq 0$, and in particular for $\alpha = f(1)^{-1} f(x)$. We hence must have $f(x^2) f(1) \geq f(x)^2$, which is analogous to the standard inequality postulating nonnegativity of the variance of a random variable, which is itself a special case of the Cauchy--Schwarz inequality.

Now given a monotone semiring homomorphism $f : F \to \R_+$, its linear extension, which we also denote by $f$, is clearly in $C^*$. To establish extremality, suppose that $f = r f_1 + (1-r) f_2$ for $r\in (0,1)$ and nonzero $f_1,f_2 \in C^*$. We assume normalization such that $f_1(1) = f_2(1) = 1$, using the inequalities $f_1(1) > 0$ and $f_2(1) > 0$ derived above for all nonzero elements of $C^*$. Then we have, per the above,
\[
	f(x)^2 = f(x^2) = r f_1(x^2) + (1 - r) f_2(x^2) \geq r f_1(x)^2 + (1 - r) f_2(x)^2 .
\]
Also expanding $f(x)^2$ and cancelling terms therefore results in the inequality
\[
	2 f_1(x) f_2(x) \geq f_1(x)^2 + f_2(x)^2.
\]
This is true if only if $f_1(x) = f_2(x) = f(x)$, as was to be shown: $f$ is extremal.

This finishes the proof in the case where $F$ is order-cancellative. For arbitrary $F$, we can reduce to this case upon equipping $F$ with the new preorder relation in which $x$ is greater than or equal to $y$ if and only if
\[
	\exists w\in F, \:\: x + w \geq y + w.
\]
This makes $F$ order-cancellative by construction. Therefore by what we have already shown, condition~\ref{hom_order_sf} implies that for every $\eps > 0$ there are $m\in\N$ and $w\in F$ with
\[
	 x + \sum_{j=0}^m \eps^{j+1} u^j + w \geq y + w.
\]
We can now use the following standard argument to get rid of $w$. By induction and chaining inequalities---see the proof of~\cite[Theorem~6.18]{ocm} and references thereafter---we can prove that for every $n\in\Nplus$,
\[
	n\left( x + \sum_{j=0}^m \eps^{j+1} u^j \right) + w \geq ny + w.
\]
Thus choosing $k$ with $u^k \geq w$ and taking $n \geq \eps^{-(k+1)}$, we get
\[
	x + \sum_{j=0}^m \eps^{j+1} u^j + \eps^{k+1} u^k \geq y + w \geq y,
\]
so that~\eqref{approx_order_ineq} follows again upon an adjustment to $\eps$. 
\end{proof}

The main nontrivial ideas have entered in this proof. Deriving \Cref{main} from \Cref{semifield_case} is now relatively straightforward.

\begin{proof}[Proof of \Cref{main}]
The simple implications from either of the third three conditions to \ref{hom_order} are now routine, upon choosing $r := f(u)$, applying $f$ to both sides of the resulting inequality, suitably rearranging, and taking the limit $\eps \to 0$. In order to be able to divide by $f(z)$, one needs to use $f(z) > 0$ for nonzero $z$. This follows from the assumption $f(u) \geq f(1) > 0$ together with $u^k z \geq 1$ for suitable $k\in\N$.

For the implications from \ref{hom_order} to \ref{closure_order} and \ref{cat_order}, we use \Cref{semifield_case}, applied with $F$ being the semifield of fractions generated by $S$. This makes sense as soon as $S$ is zero-divisor-free~\cite[Example~11.7]{golan}; but this we can assume without loss of generality, since if $xy = 0$ for nonzero $x,y\in S$, then we choose $k\in\N$ with $u^k x \geq 1$ and $u^k y \geq 1$, giving $0 = u^{2k} x y \geq 1$, so that the preorder on $S$ is trivial in that $0 \geq x$ for all $x\in S$, making all statements of \Cref{main} trivially true. The polynomially universal element $u\in S$ is still polynomially universal in $F$.
	
We equip this semifield with the catalytic preorder, meaning that $xy^{-1} \geq ab^{-1}$ for nonzero elements of $F$ if and only if there is nonzero $t\in S$ with $xbt \geq ayt$, which is easily seen to be well-defined. Then the condition~\ref{hom_order_sf} of \Cref{semifield_case} is clearly equivalent to the present condition~\ref{hom_order}. We thus only need to deduce the conditions \ref{closure_order} and \ref{cat_order} from their counterparts in \Cref{semifield_case}.
	
We first show that condition~\ref{closure_order} in $S$ follows from condition~\ref{closure_order_sf2} in the semifield of fractions $F$. Given~\ref{closure_order_sf2} in $F$, we choose $p$ such that $p(r) \leq \eps$. Then writing the inequality~\eqref{approx_order_ineq2} explicitly in terms of fractions shows that there is nonzero $z \in S$ with
\[
	n\, z\, x + n\, p(u)\, z \geq n\, z\, y,
\]
where $n$ is chosen as a multiple of all demoninators of the coefficients of $p$, so that our new polynomial is $n\, p$.

If \ref{cat_order_sf} of \Cref{semifield_case} holds in the semifield of fractions, then the new condition~\ref{cat_order} trivially follows upon multiplying by $m\in\N$ so as to make all coefficients integral.

It remains to prove the implication from~\ref{cat_order} to~\ref{asymp_order}. For given $p$, $m$ and $z$, we get by the standard argument of chaining inequalities that for every $n\in\N$,
\[
	p(u)^n\, z\, x^n \geq m^n\, z\, y^n.
\]
Now we choose $k\in\N$ such that $u^k \geq z$ and $u^k z \geq 1$, which gives
\[
	u^{2k}\, p(u)^n \, x^n \geq z\, u^k\, m^n \, y^{n} \geq m^n \, y^n.
\]
So if we choose $n$ large enough so that $p(r)^n \geq r^{2k}$, then  we can choose $\tilde{p} := X^{2k}\, \hat{p}^n$ and $\tilde{m} := m^n$, which results in $\tilde{p}(u) \, x^n \geq \tilde{m} \, y^n$ as desired, and also
\[
	\tilde{p}(r) = r^{2k} p(r)^n \leq \hat{p}(r)^{2n} \leq (1 + 2\eps)^{2n} \leq (1 + 5\eps)^n,
\]
where the final inequality holds for sufficiently small $\eps$. This is enough since $\eps$ was arbitrary.
\end{proof}

\section{First applications and comparison with other Positivstellens\"atze}
\label{compare}

We start by instantiating \Cref{main} to polynomial semirings, where it is strictly weaker than P\'olya's Positivstellensatz. We then state some more interesting abstract special cases involving ordered rings and quadratic modules.

\begin{ex}
\label{poly}
Among the simplest preordered semirings are the polynomial semirings $\N[X_1,\ldots,X_d] = \N[\underline{X}]$ with the coefficientwise preorder. This $S$ clearly has characteristic zero. While the element $u := 1 + \sum_i X_i$ satisfies the upper bound part of polynomial universality, it satisfies the lower bound property only on those polynomials which have strictly positive constant coefficient; together with zero, these polynomials form a subsemiring $S := \N[\underline{X}]'$ to which we can apply \Cref{main}. Similarly, we may consider $S := \R_+[X_1,\ldots,X_d]' = \R_+[\underline{X}]'$, the semiring of real polynomials with nonnegative coefficients and strictly positive constant coefficient together with zero. Again using $u = 1 + \sum_i X_i$, the hypotheses of \Cref{main} are satisfied.

The monotone homomorphisms $\N[\underline{X}]' \to \R_+$, and similarly the ones $\R_+[\underline{X}]' \to \R_+$, are precisely the evaluation maps at any point in the nonnegative orthant $s\in\R_+^d$. One way to see this is to use the fact that the only additive monotone maps $\R\to\R$ are the $\R$-linear ones (Cauchy functional equation). We then obtain that for nonzero polynomials $f,g\in\R_+[\underline{X}]$, the following statements are equivalent\footnote{The case $f,g \in \N_+[\underline{X}]$ is not significantly different.}, with all inequalities between polynomials referring to the coefficientwise order:
\begin{enumerate}
	\item\label{a} $f(s) \geq g(s)$ for every $s\in\R^d_+$.
	\item\label{b} For every $r\in\R_+$ and $\eps > 0$, there exist a single-variable polynomial $p\in\R_+[X]$ and a nonzero polynomial $h \in \R_+[\underline{X}]$ such that $p(r) \leq \eps$ and
	\[
		h f + h\, p\left( 1 + \sum_i X_i \right) \geq h \, g.	
	\]
\item For every $r\in\R_+$ and $\eps > 0$, there exist a single-variable polynomial $p\in\R_+[X]$ and a nonzero polynomial $h \in \R_+[\underline{X}]$ such that $p(r) \leq 1 + \eps$ and
	\[
		p\left( 1 + \sum_i X_i \right) \, h f \geq h \, g.
	\]
	\item For every $r\in \R_+$ and $\eps > 0$, there exist a single-variable polynomial $p\in\R_+[X]$ and $n\in\N$ such that $p(r) \leq (1 + \eps)^n$ and
	\[
		p\left( 1 + \sum_i X_i \right) \, f^n \geq g^n.
	\]
\end{enumerate}
A priori we obtain these results only for when both $f$ and $g$ have strictly positive constant coefficient, but it is straightforward to see that this is irrelevant upon adding a small constant coefficient to any given $f,g \in \R_+[\underline{X}]$.

As the proof of \Cref{main} shows, the difficult implication here is the one from~\ref{a} to~\ref{b}. As we explain now, this implication also follows from P\'olya's Positivstellensatz, which moreover gives a concrete form for which $h$ and $p$ one can take in~\ref{b}.

To recall P\'olya's Positivstellensatz, we write
\[
	\Delta_d := \left\{ s\in \R_+^{d+1} \: \bigg| \: \sum_j s_j = 1 \right\}
\]
for the standard $d$-simplex, and consider a homogeneous polynomial $\hat{q} \in \R[X_0,\ldots,X_d]$. Let us also write $\hat{u} := \sum_{j=1}^{d+1} X_j$. Then one way to formulate P\'olya's Positivstellensatz is to say that the following are equivalent:
\begin{enumerate}[label=(\roman*)]
	\item $\hat{q}(s) \geq 0$ for every $s \in \Delta_d$.
	\item For every $\delta > 0$, there is $k \in \N$ such that $\hat{u}^k (\hat{q} + \delta \hat{u}^{\deg(\hat{q})}) \geq 0$ coefficientwise.
\end{enumerate}
We can use this result to study the positivity of an arbitrary nonzero polynomial $q \in \R[X_1,\ldots,X_d]$ on the orthant $\R_+^d$ by applying it to the homogenization $\hat{q} \in \R[X_0,\ldots,X_d]$. Since
\[
	q(s) \geq 0 \quad \forall s \in \R_+^d \qquad \Longleftrightarrow \qquad \hat{q}(s) \geq 0 \quad \forall s \in \Delta_d,
\]
we obtain that the following are equivalent for every $q \in \R[\underline{X}]$, with $u = 1 + \sum_i X_i$:
\begin{enumerate}[label=(\roman*)]
	\item\label{i} $q(s) \geq 0$ for every $s \in \R_+^d$.
	\item\label{ii} For every $\delta > 0$, there is $k \in \N$ such that $u^k (q + \delta u^{\deg(q)}) \geq 0$ coefficientwise.
\end{enumerate}
The key observation here is that $\hat{u}^k (\hat{q} + \eps \hat{u}^{\deg(\hat{q})})$ is the homogenization of $u^k (q + \eps u^{\deg(q)})$, so that each one of these polynomials has nonnegative coefficients if and only if the other one does. Note that the above equivalence is very similar to~\cite[Corollary~2.5]{marshall2}.

Now suppose that we are given $f,g\in\R_+[\underline{X}]$ satisfying our~\ref{a}. Consider $q := f - g$. Applying the equivalence of~\ref{i} and~\ref{ii} then shows that also~\ref{b} holds, namely with $p = \delta X^{\deg(f - g)}$ and $h = u^k$.
\end{ex}

So due to being weaker than P\'olya's Positivstellensatz, our \Cref{main} is not of much interest in the study of positive polynomials on the orthant $\R_+^d$. However, our result is much more general and also provides a characterization of positivity of polynomials on \emph{arbitrary subsets} of the orthant, as we show now.

\begin{ex}
	\label{ex_subset}
	Let $I \subseteq \R[\underline{X}]$ be any ideal, for example the ideal of polynomials which vanish on a given subset of $\R_+^d$. Consider the zero locus within the nonnegative orthant,
	\[
		\mathcal{V}_+(I) := \{ s \in \R_+^d \mid g(s) = 0 \;\,\forall g \in I \}.
	\]
	Then the following are equivalent for any $f \in \R[\underline{X}]$:
	\begin{enumerate}
		\item $f(s) \geq 0$ for all $s \in \mathcal{V}_+(I)$.
		\item For every $r \in \R_+$ and $\eps > 0$, there exist a single-variable polynomial $p \in \R_+[X]$, a nonzero polynomial $h \in \R_+[\underline{X}]$ and $g \in I$ such that $p(r) \leq \eps$ and the polynomial
			\[
				h (f + p(u)) + g 
			\]
			has nonnegative coefficients, where $u = 1 + \sum_j X_j$.
	\end{enumerate}
	To see this, we consider the semiring $R_+[\underline{X}]'$ equipped with the preorder in which $q_1 \ge q_2$ if and only if $(q_1 - q_2) \in (\R_+[\underline{X}] + I)$, which gives a preordered semiring with polynomially universal element $u$. Writing a given $f \in \R[\underline{X}]$ as $f = f_+ - f_-$ with $f_+,f_- \in \R_+[\underline{X}]'$ then makes \Cref{main}\ref{hom_order}--\ref{closure_order} specialize to the equivalence claimed above.
\end{ex}

The next example provides a useful permanence property of our polynomial growth condition.

\begin{ex}
\label{polyS}
Let $S$ be a preordered semiring of characteristic zero having a polynomially universal element $u$, so that the assumptions of \Cref{main} are satisfied. Then the semiring of polynomials with nonzero constant coefficient $S[X]'$, equipped with the coefficientwise preorder, is again a preordered semiring of characteristic zero, and has a polynomially universal element given by $u + X$. Hence the assumptions of \Cref{main} are stable under $S\mapsto S[X]'$.

The same statement apply to the semiring of Laurent polynomials $S[X,X^{-1}]$ with the coefficientwise order, using $X^{-1} + u + X$ as a polynomially universal element. Now the restriction to nonzero constant coefficient is no longer needed.
\end{ex}

Next, we specialize our \Cref{main} to ordered rings and quadratic modules. In the upcoming \Cref{qm,kd-like}, we will not make use of the characterization \ref{asymp_order} involving large powers of the semiring elements.

The following result is similar to the Positivstellensatz of Krivine--Kadison--Dubois\footnote{See~\cite[Proposition~3]{krivine} and~\cite{dubois}, and also~\cite{BS} for a modern formulation with purely algebraic proof, and~\cite[Theorem~5.4.4]{marshall} for a textbook treatment of a very general version.}, but replaces its Archimedeanicity assumption by our polynomial growth condition, at the cost of a weaker conclusion: in Krivine--Kadison--Dubois, it is possible to take $w \in \N$.

\begin{thm}
\label{kd-like}
Let $R$ be a ring of characteristic zero and $P \subseteq R$ a subsemiring. Suppose that there is $v \in P$ such that for every $a\in R$, we can find $p\in\N[X]$ with
\[
	p(v) - a \in P. 
\]
Then for every $a\in R$, the following are equivalent:
\begin{enumerate}
\item\label{Ppos} $f(a) \geq 0$ for every ring homomorphism $f : R \to \R$ with $f(P) \subseteq \R_+$.
\item For every $r\in\R_+$ and $\eps > 0$, there exist a polynomial $q\in\N[X]$, an element $w\in P$ and $n\in\Nplus$ such that $q(r) \leq \eps n$ and
\[
	(1 + w)\, (n\, a + q(v)) \in P.
\]
\end{enumerate}
\end{thm}

\begin{proof} 
The existence of such $v\in P$ implies that $R = P - P$.

We define the semiring $S$ to be given by zero together with those $x\in P$ for which there is $n\in\Nplus$ with $(nx - 1) \in P$. In particular, $S$ contains every element of the form $1 + x$ for $x\in P$. This $S$ is closed under addition, since
\[
	mn(x + y) - 1 = m(nx - 1) + n(my - 1) + (m + n - 1),
\]
where the last term is in $P$ because of $1 \in P$. This $S$ is also closed under multiplication, thanks to
\[
	mnxy - 1 = (nx - 1)\, (my - 1) + (nx - 1) + (my - 1).
\]
Declaring $x \geq y$ to hold in $S$ if and only if $(x - y) \in P$ turns $S$ into a preordered semiring. It is straightforward to check that $u := 1 + v$ is a polynomially universal element, where the upper bound is by assumption and the lower bound works as follows: if $x\in S$ is nonzero, then we have $n\in\Nplus$ with $(nx - 1) \in P$. Using the polynomial $p := n\, X$ now obviously works, since $u \geq 1$.

We now apply \Cref{main} to this $S$. The assumptions imply that restriction from $R$ to $S$ implements a bijection between the monotone homomorphisms $f : S\to \R_+$ and the ring homomorphisms $f : R \to \R$ with $f(P) \subseteq \R_+$. Hence we conclude that~\ref{Ppos} holds if and only if for every $\eps > 0$ and $r\in\R_+$ there exist nonzero $z\in S$ and $p\in\N[X]$ and $m\in\Nplus$ with $p(1 + r) \leq \eps\, m$ and $z(m\, a + p(1 + v)) \in P$. Since $z\in S$, we can conclude that $w := \ell z - 1 \in P$ for suitable $\ell\in\Nplus$. We now take $q := p(1 + X)$ and the claim follows.
\end{proof}

For a ring $A$, we write $A^2 \subseteq A$ for its subsemiring consisting of all sums of squares. A \emph{quadratic module} $M\subseteq A$ is a subset with $1\in M$ which is closed under addition as well as under multiplication by elements from $A^2$. We turn $A^2$ into a preordered semiring by putting $x \geq y$ if and only if $x - y \in M$. Suppose that $A$ has characteristic zero and that $A= A^2 - A^2$, meaning that every element is a difference of sums of squares. Then in order for Strassen's \Cref{spss} to even apply, $M$ must be Archimedean. But thanks to \Cref{main}, we can now state a Positivstellensatz for quadratic modules which weakens this Archimedeanicity to polynomial growth. In the following theorem, we restrict ourselves to the case of $\R$-algebras for simplicity.

\begin{thm}
\label{qm}
Let $A$ be an $\R$-algebra and $A^2 \subseteq A$ the subsemiring of sums of squares. Let $M \subseteq A$ be a quadratic module for which there is $v\in A^2$ such that for every $a\in A^2$, we can find $p\in\R_+[X]$ with
\[
	p(v) - a \in M.
\]
Then the following are equivalent for $a \in A$:
\begin{enumerate}
\item\label{qmpos} $f(a) \geq 0$ for every algebra homomorphism $f : A \to \R$ with $f(M) \subseteq \R_+$.
\item For every $r\in\R_+$ and $\eps > 0$, there exist a polynomial $q\in\R_+[X]$ and an element $w\in A^2$ such that $q(r)\leq \eps$ and
\[
	(1 + w)\, (a + q(v)) \in M.
\]
\end{enumerate}
\end{thm}

The reduction to \Cref{main} is similar to the proof of \Cref{kd-like} above.

\begin{proof}
	It is well-known that $A = A^2 - A^2$, for example since $a = \left(\tfrac{1+a}{2}\right)^2 - \left(\tfrac{1-a}{2}\right)^2$~\cite[p.~22]{marshall}.

We apply \Cref{main} with $S$ being given by zero together with those $x \in A$ for which there is $\eps > 0$ with $(x - \eps) \in A^2$. This set is clearly closed under addition, and also under multiplication since $xy - \eps\delta = (x - \eps)(y - \delta) + \eps (y - \delta) + (x - \eps) \delta$. We order $S$ by declaring $x \geq y$ to hold if and only if $(x - y) \in M$, and it is easy to see that we get a preordered semiring. Then $u := 1 + v$ is a polynomially universal element, where the upper bound is by assumption and the lower bound works as follows: if $x\in S$ is nonzero, then we have $\eps > 0$ with $(x - \eps) \in A^2$. Using the polynomial $p := \eps^{-1} X$, we have $x\, p(u) = \eps^{-1}\, x\, (1 + v) \geq \eps^{-1}\, x \geq 1$, as was to be shown.

We now apply \Cref{main} together with \Cref{Ralg} to this $S$ with $x,y\in S$ chosen such that $a = x - y$. The assumptions imply that restriction from $A$ to $S$ implements a bijection between the monotone homomorphisms $f : S \to \R_+$ and the algebra homomorphisms $f : A \to \R$ with $f(M) \subseteq \R_+$. Hence we conclude that~\ref{qmpos} holds if and only if for every $\eps > 0$ and $r\in\R_+$ there exist nonzero $z\in S$ and $p\in\R_+[X]$ with $p(1 + r) \leq \eps$ and $z (a + p(1 + v)) \in M$. Since $z\in S$, we can assume $w := z - 1 \in A^2$ after rescaling of $z$. We now take $q := p(1 + X)$ and the claim follows.
\end{proof}

\bibliographystyle{plain}
\bibliography{strassen_positivstellensatz}

\end{document}